\documentclass[11pt]{article}

\usepackage[margin=1in]{geometry}
\usepackage{amsmath,amssymb,amsthm}
\usepackage{booktabs}
\usepackage{hyperref}

\newtheorem{theorem}{Theorem}[section]
\newtheorem{proposition}[theorem]{Proposition}
\newtheorem{lemma}[theorem]{Lemma}
\newtheorem{corollary}[theorem]{Corollary}
\theoremstyle{remark}
\newtheorem{remark}[theorem]{Remark}

\newcommand{\area}{\Delta}
\newcommand{\arsinh}{\operatorname{arsinh}}

\title{\textbf{Average Chord Lengths in a Triangle}}
\author{Stanley Rabinowitz\\
545 Elm St Unit 1, Milford, New Hampshire 03055, USA\\
\texttt{stan.rabinowitz@comcast.net}}
\date{}

\begin{document}

\maketitle

\begin{abstract}
Let $P$ be a point inside a triangle $T$.  We consider the average
length of the chords of $T$ through $P$, where the direction of the
chord is chosen uniformly.  An elementary formula is obtained in terms
of the distances from $P$ to the sides and vertices of the triangle.
Several classical triangle centers give especially simple
specializations.  For example, if $I$ is the incenter, then
\[
 M_T(I)=\frac{2r}{\pi}
 \log\left(\cot\frac A4\cot\frac B4\cot\frac C4\right).
\]
Our main result is the sharp inequality
\[
 M_T(P)\le
 \frac{p}{\pi\sqrt3}\log(2+\sqrt3),
\]
valid simultaneously for every triangle of perimeter $p$ and every
interior point $P$.  Thus, among all such pairs $(T,P)$, the largest
possible average chord length occurs only when $T$ is equilateral
and $P$ is its center.  The proof is an elementary symmetrization
argument.  We close
with brief remarks relating the problem to the radial center of a
convex body, the electrostatic potential center of a triangle, and
dual quermassintegrals.
\end{abstract}



\section{Introduction}

Fix a point $P$ inside a triangle $T$.  Through $P$ draw a line making
an angle $\theta$ with a fixed direction, and let $\ell_P(\theta)$ be
the length of the chord cut from $T$ by this line.  Since an unoriented
line has period $\pi$, it is natural to define the average chord
length through $P$ by
\[
 M_T(P)=\frac1\pi\int_0^\pi \ell_P(\theta)\,d\theta.
\]
The question is elementary to state: how does $M_T(P)$ depend on the
triangle and on the location of $P$?  This precise average-chord
question for a point in a planar convex region was also posed by
Nandakumar R.\ on MathOverflow in 2021 \cite{Nandakumar}; numerical
experiments there indicated that, for a triangle, the maximizing
point need not be either the centroid or the incenter.

The corresponding integral has appeared in several other settings.
In electrostatics it is the Coulomb potential in three-dimensional
space of a uniformly charged triangular lamina, evaluated at a point
of the lamina:
\[
 \int_T\frac{dA(X)}{|X-P|}.
\]
Explicit potential formulas for polygons, and in particular for
triangles, have long been known; see Duffin and McWhirter
\cite{DuffinMcWhirter} and, for recent formulas emphasizing the
centroid and incenter of a triangle, B\"ohm and Runge
\cite{BohmRunge}.  Our purpose here is different.  We keep the
geometric interpretation in terms of average chords, derive the
needed formulas directly by elementary trigonometry and integration,
and study their consequences for familiar triangle centers.

A first example already gives a pleasantly compact answer.

\section{A motivating example: the incenter}

Let $I$ be the incenter of $ABC$, and let $r$ be the inradius.

\begin{proposition}
The average chord length through the incenter is
\[
 {
 M_T(I)=\frac{2r}{\pi}
 \log\left(\cot\frac A4\cot\frac B4\cot\frac C4\right).
 }
\]
\end{proposition}

\begin{proof}
Consider the rays from $I$ which meet the side $BC$.  The perpendicular
distance from $I$ to $BC$ is $r$.  If a ray makes angle $\phi$ with
the perpendicular to $BC$, then its length from $I$ to $BC$ is
$r\sec\phi$.

At $B$ and $C$ the angles of triangle $IBC$ are $B/2$ and $C/2$.
Consequently the total contribution of the rays meeting $BC$ to the
integral of the radial distance is
\[
 r\left(\log\cot\frac B4+\log\cot\frac C4\right).
\]
Adding the analogous contributions from the other two sides, each of
the three logarithms occurs twice.  Dividing by $\pi$ gives the
formula.
\end{proof}

For an equilateral triangle of side $a$ this becomes
\[
 M_T(I)=
 \frac{\sqrt3\,a}{\pi}\log(2+\sqrt3).
\]
The incenter and centroid coincide in this case.

For a $3$--$4$--$5$ triangle, $s=6$ and $r=1$, so the same formula
may be written
\[
 M_T(I)
 =
 \frac{2}{\pi}
 \left(\arsinh 3+\arsinh 2+\arsinh 1\right)
 \approx 2.63781.
\]
Thus even for a very familiar triangle the average chord length is
not an especially obvious classical length.

An equivalent formula for the associated potential integral at the
incenter can be extracted from the formulas of B\"ohm and Runge
\cite{BohmRunge}.  The derivation above is included because it shows
directly why the quarter-angles occur.

\section{A formula for an arbitrary interior point}

Let $\rho_P(\theta)$ denote the distance from $P$ to the boundary of
$T$ along the ray in direction $\theta$.

\begin{lemma}\label{lem:radial}
For every interior point $P$,
\[
 {
 M_T(P)=\frac1\pi\int_0^{2\pi}\rho_P(\theta)\,d\theta.
 }
\]
Moreover,
\[
 {
 \pi M_T(P)=
 \int_T\frac{dA(X)}{|X-P|}.
 }
\]
\end{lemma}

\begin{proof}
The chord through $P$ in direction $\theta$ has length
\[
 \ell_P(\theta)=\rho_P(\theta)+\rho_P(\theta+\pi).
\]
Therefore
\[
 \int_0^\pi\ell_P(\theta)\,d\theta
 =
 \int_0^{2\pi}\rho_P(\theta)\,d\theta.
\]
For the second identity, use polar coordinates centered at $P$:
\[
 \int_T\frac{dA(X)}{|X-P|}
 =
 \int_0^{2\pi}\int_0^{\rho_P(\theta)}
 \frac1r\,r\,dr\,d\theta
 =
 \int_0^{2\pi}\rho_P(\theta)\,d\theta.
\]
\end{proof}

Write
\[
 d_a=\operatorname{dist}(P,BC),\qquad
 d_b=\operatorname{dist}(P,CA),\qquad
 d_c=\operatorname{dist}(P,AB).
\]

\begin{theorem}\label{thm:general}
For every interior point $P$,
\[
 {
 M_T(P)=\frac1\pi\sum_{\rm cyc}
 d_a\log\frac{PB+PC+a}{PB+PC-a}.
 }
\]
Equivalently,
\[
 M_T(P)=\frac{2}{\pi}\sum_{\rm cyc}
 d_a\,\operatorname{artanh}\frac{a}{PB+PC}.
\]
\end{theorem}

\begin{proof}
Consider again the rays from $P$ which meet $BC$.  Put
\[
 \beta=\angle PBC,\qquad \gamma=\angle PCB.
\]
Integrating $d_a\sec\phi$ over the angular interval from the ray $PB$
to the ray $PC$ gives
\[
 d_a\left(\log\cot\frac{\beta}{2}
          +\log\cot\frac{\gamma}{2}\right).
\]
Let
\[
 \sigma=\frac{a+PB+PC}{2}
\]
be the semiperimeter of triangle $PBC$.  The standard half-angle
formulas give
\[
 \cot\frac{\beta}{2}\cot\frac{\gamma}{2}
 =
 \frac{\sigma}{\sigma-a}
 =
 \frac{PB+PC+a}{PB+PC-a}.
\]
Thus the contribution from the rays meeting $BC$ is
\[
 d_a\log\frac{PB+PC+a}{PB+PC-a}.
\]
Summing over the three sides and applying Lemma~\ref{lem:radial}
proves the first formula.  The second follows from
\[
 \log\frac{1+t}{1-t}=2\operatorname{artanh}t.
\]
\end{proof}

\section{Some classical triangle centers}

Theorem~\ref{thm:general} can be specialized to any interior triangle
center.  We record several cases which simplify naturally.

Let $G$ be the centroid, $K$ the symmedian point, $N$ the Nagel point,
$G_e$ the Gergonne point, and $M$ the Mittenpunkt.  As usual, let
$m_a,m_b,m_c$ denote the medians and $h_a,h_b,h_c$ the altitudes.

\subsection{The centroid}

Since
\[
 GB=\frac23m_b,\qquad
 GC=\frac23m_c,\qquad
 d_a(G)=\frac{h_a}{3},
\]
Theorem~\ref{thm:general} gives the following compact expression.

\begin{proposition}\label{prop:centroid}
\[
 {
 M_T(G)=
 \frac1{3\pi}\sum_{\rm cyc}
 h_a\log
 \frac{2m_b+2m_c+3a}
      {2m_b+2m_c-3a}.
 }
\]
\end{proposition}

\subsection{The symmedian point}

Put
\[
 S_2=a^2+b^2+c^2.
\]
The barycentric coordinates of $K$ are
\[
 (a^2:b^2:c^2).
\]
Hence
\[
 d_a(K)=\frac{2\area\,a}{S_2}.
\]
A direct vector calculation also gives
\[
 KA=\frac{2bc\,m_a}{S_2},
\]
and cyclically.

\begin{proposition}\label{prop:symmedian}
\[
 {
 M_T(K)=
 \frac{2\area}{\pi S_2}
 \sum_{\rm cyc}
 a\log
 \frac{S_2+2(cm_b+bm_c)}
      {2(cm_b+bm_c)-S_2}.
 }
\]
\end{proposition}

The denominator in each logarithm is positive.  For example,
\[
 2(cm_b+bm_c)>S_2,
\]
which is just the triangle inequality $KB+KC>a$ after multiplying
by $S_2/a$.

\subsection{Nagel, Gergonne, and Mittenpunkt}

It is convenient to put
\[
 x=s-a,\qquad y=s-b,\qquad z=s-c.
\]
The following table lists barycentric coordinates and the distance to
$BC$.  The other sideline distances are obtained cyclically.

\begin{center}
\begin{tabular}{lll}
\toprule
Point & Barycentric coordinates & $d_a$\\
\midrule
Incenter $I$ & $(a:b:c)$ & $r$\\[1mm]
Centroid $G$ & $(1:1:1)$ & $h_a/3$\\[1mm]
Symmedian point $K$ & $(a^2:b^2:c^2)$
 & $\dfrac{2\area a}{a^2+b^2+c^2}$\\[3mm]
Nagel point $N$ & $(x:y:z)$
 & $\dfrac{2rx}{a}$\\[3mm]
Gergonne point $G_e$ & $(yz:zx:xy)$
 & $\dfrac{2s\,yz}{a(4R+r)}$\\[3mm]
Mittenpunkt $M$ & $(ax:by:cz)$
 & $\dfrac{s x}{4R+r}$\\
\bottomrule
\end{tabular}
\end{center}

Thus Theorem~\ref{thm:general} immediately gives explicit formulas
for the average chord through each of these points.  For example, the
Mittenpunkt satisfies
\[
 {
 M_T(M)=
 \frac{s}{\pi(4R+r)}
 \sum_{\rm cyc}
 (s-a)
 \log\frac{MB+MC+a}{MB+MC-a}.
 }
\]
The table is often more useful than expanding the remaining vertex
distances into radicals.

\subsection{A $3$--$4$--$5$ example}

For a concrete comparison, the following are the average chord
lengths through these six familiar centers in a $3$--$4$--$5$
triangle.  The values are obtained directly from
Theorem~\ref{thm:general}.

\begin{center}
\begin{tabular}{lc}
\toprule
Point & Average chord length\\
\midrule
Centroid $G$ & $2.63841$\\
Incenter $I$ & $2.63781$\\
Symmedian point $K$ & $2.57012$\\
Mittenpunkt $M$ & $2.55819$\\
Gergonne point $G_e$ & $2.55035$\\
Nagel point $N$ & $2.42950$\\
\bottomrule
\end{tabular}
\end{center}

The centroid and incenter values are remarkably close, although the
two points are distinct.  By comparison, the sharp bound proved in
the next section gives $2.90431\ldots$ for an equilateral triangle
with the same perimeter $12$.

\section{A sharp bound for every interior point}

We now prove the main result.  Remarkably, the same sharp bound holds
for every interior point, not merely for a distinguished center.

\begin{theorem}\label{thm:main}
Let $T$ be a triangle of perimeter $p$, and let $P$ be any interior
point of $T$.  Then
\[
 {
 M_T(P)\le
 \frac{p}{\pi\sqrt3}\log(2+\sqrt3).
 }
\]
Equality holds if and only if $T$ is equilateral and $P$ is its
center.
\end{theorem}

The proof uses only an elementary form of Steiner symmetrization.  We
include the details.

\begin{lemma}\label{lem:symm}
Fix a side $BC$ of a triangle $T$ and a point $P\in T$.  Replace
each segment of $T$ parallel to $BC$ by a segment of the same length
whose midpoint lies on the line through $P$ perpendicular to $BC$.
The resulting set $T^*$ is an isosceles triangle with the same base
length and altitude as $T$, and $P\in T^*$.  Moreover,
\[
 \int_{T^*}\frac{dA(X)}{|X-P|}
 \ge
 \int_T\frac{dA(X)}{|X-P|},
\]
and
\[
 p(T^*)\le p(T).
\]
Equality in the first inequality holds only when $T$ is already
symmetric about the line through $P$ perpendicular to $BC$.
\end{lemma}

\begin{proof}
Take $BC$ horizontal and write $P=(0,t)$.  At height $y$, the
horizontal section of $T$ is an interval $I_y$ of some length
$\ell(y)$.  The corresponding section of $T^*$ is the centered
interval
\[
 I_y^*=\left[-\frac{\ell(y)}2,\frac{\ell(y)}2\right].
\]
Since the section at the height of $P$ is centered at $P$, we have
$P\in T^*$.  The section lengths are unchanged, so $T^*$ has the same
base length and altitude as $T$.

For fixed $y\ne t$, the function
\[
 x\longmapsto\frac1{\sqrt{x^2+(y-t)^2}}
\]
is even and strictly decreasing as $|x|$ increases.  Hence its
integral over an interval of prescribed length is largest when the
interval is centered at the origin.  Integrating over $y$ gives the first inequality.  The exceptional
level $y=t$ has measure zero.  If desired, the harmless singularity
there can be avoided by first replacing the kernel by
\[
 \frac{1}{\sqrt{x^2+(y-t)^2+\varepsilon^2}}
\]
and then letting $\varepsilon\downarrow0$.  The same argument applies
when $P$ lies on the boundary.

If equality holds, then for almost every $y$ the midpoint of $I_y$
must lie on $x=0$.  The midpoints of the parallel sections of a
triangle vary affinely with $y$, so the original triangle must itself
be symmetric about $x=0$.

For the perimeter assertion, let $a=BC$ and let $h$ be the altitude
from the opposite vertex.  If $\delta$ is the horizontal displacement
of that vertex from the midpoint of $BC$, the sum of the two
non-base sides is
\[
 f(\delta)=
 \sqrt{h^2+\left(\delta+\frac a2\right)^2}
 +
 \sqrt{h^2+\left(\delta-\frac a2\right)^2}.
\]
The function $f$ is even and strictly convex, and therefore has its
minimum at $\delta=0$.  The symmetrized triangle has $\delta=0$,
which proves $p(T^*)\le p(T)$, with equality if and only if
$AB=AC$.  Notice that this perimeter comparison depends only on the
triangle; it is independent of the location of $P$.
\end{proof}

We also need a simple compactness observation.

\begin{lemma}\label{lem:area-bound}
For every triangle $T$ of area $\area$ and every point $P\in T$,
\[
 \int_T\frac{dA(X)}{|X-P|}
 \le 2\sqrt{\pi\area}.
\]
\end{lemma}

\begin{proof}
By Lemma~\ref{lem:radial},
\[
 V:=\int_T\frac{dA(X)}{|X-P|}
 =\int_0^{2\pi}\rho_P(\theta)\,d\theta,
\]
while
\[
 2\area=\int_0^{2\pi}\rho_P(\theta)^2\,d\theta.
\]
Cauchy--Schwarz therefore gives
\[
 V^2
 \le
 2\pi\int_0^{2\pi}\rho_P(\theta)^2\,d\theta
 =
 4\pi\area.
\]
\end{proof}

\begin{proof}[Proof of Theorem~\ref{thm:main}]
Because both perimeter and $M_T(P)$ scale linearly, fix the perimeter
$p$.  It is convenient here to allow $P$ temporarily to lie on the
boundary of $T$; the integral
\[
 V(T,P)=\int_T\frac{dA(X)}{|X-P|}
\]
is still finite.

A maximizing pair $(T,P)$ exists.  Indeed, after translating $P$ to
the origin, all vertices lie in a fixed bounded disk.  A maximizing
sequence therefore has a convergent subsequence.  A degenerate
limiting triangle has area tending to zero, and
Lemma~\ref{lem:area-bound} then forces $V(T,P)\to0$, so a maximizer
cannot be degenerate.  For a nondegenerate limit, continuity follows
by translating $P$ to the origin: the indicator functions of the
triangles converge almost everywhere in a common bounded disk, while
$1/|X|$ is locally integrable in the plane.  Thus the maximum is
attained on a nondegenerate triangle.  The argument below will in
fact force the maximizing point to be the interior center of an
equilateral triangle.

Choose any side $BC$ and form the triangle $T^*$ of
Lemma~\ref{lem:symm}.  Then
\[
 V(T^*,P)\ge V(T,P),
 \qquad
 p(T^*)\le p.
\]
Set
\[
 \lambda=\frac{p}{p(T^*)}\ge1
\]
and enlarge $T^*$ by a homothety centered at $P$ with factor
$\lambda$.  Since
\[
 V(\lambda T^*,P)=\lambda V(T^*,P)
 \ge \lambda V(T,P),
\]
maximality forces $\lambda=1$ and
\[
 V(T^*,P)=V(T,P).
\]
Thus equality holds throughout Lemma~\ref{lem:symm}.  The perimeter
equality gives $AB=AC$, while equality in the integral shows that
$T$ is symmetric about the line through $P$ perpendicular to $BC$;
in particular, $P$ lies on the perpendicular bisector of $BC$.

Apply the same argument to a second side.  The triangle is then
isosceles with respect to two different sides, and hence equilateral.
The two symmetry axes meet at its center, so $P$ is that center.

It remains only to evaluate $M_T(P)$ for an equilateral triangle.
Let its side length be $a=p/3$.  The distance from its center to a
side is
\[
 d=\frac{a\sqrt3}{6}=\frac{p}{6\sqrt3}.
\]
For the rays meeting one side,
\[
 \rho(\theta)=d\sec\theta,
 \qquad -\frac{\pi}{3}\le\theta\le\frac{\pi}{3}.
\]
Thus
\[
 \int_0^{2\pi}\rho(\theta)\,d\theta
 =
 3d\int_{-\pi/3}^{\pi/3}\sec\theta\,d\theta
 =
 6d\log(2+\sqrt3)
 =
 \frac{p}{\sqrt3}\log(2+\sqrt3).
\]
Division by $\pi$ completes the proof.
\end{proof}

\begin{corollary}
For every triangle of perimeter $p$, the average chord through each
of the points
\[
 I,\quad G,\quad K,\quad G_e,\quad N,\quad M
\]
is at most
\[
 \frac{p}{\pi\sqrt3}\log(2+\sqrt3),
\]
with equality only for an equilateral triangle.
\end{corollary}

\begin{remark}
For the incenter alone, the sharp inequality also has a short direct
proof.  Put
\[
 x=\frac A2,\qquad y=\frac B2,\qquad z=\frac C2,
 \qquad x+y+z=\frac{\pi}{2}.
\]
Since
\[
 s=r(\cot x+\cot y+\cot z),
\]
the incenter formula reduces the assertion to
\[
 \frac{\sum\log\cot(x/2)}{\sum\cot x}
 \le
 \frac{\log(2+\sqrt3)}{\sqrt3}.
\]
This follows immediately from Jensen's inequality applied to
\[
 h(x)=
 \log\cot\frac x2-c\cot x,
 \qquad
 c=\frac{\log(2+\sqrt3)}{\sqrt3}.
\]
Indeed,
\[
 h''(x)=
 \csc x\,\cot x\bigl(1-2c\csc x\bigr)<0
 \qquad (0<x<\pi/2),
\]
because $2c>1$ and $\csc x\ge1$.  Thus $h$ is strictly concave, and
$h(\pi/6)=0$.
\end{remark}

\section{The point of maximum average chord}

For a fixed triangle, it is natural to ask which point $P$ maximizes
$M_T(P)$.  By Lemma~\ref{lem:radial}, this is exactly the problem of
maximizing
\[
 \int_T\frac{dA(X)}{|X-P|}.
\]
Abraham and Kova\v{c} \cite{AbrahamKovac} studied this point as the
center of maximal electrostatic potential of a uniformly charged
triangle.  It is catalogued as $X(5626)$ in Kimberling's
\emph{Encyclopedia of Triangle Centers} \cite{KimberlingETC}; see
also Nicollier \cite{Nicollier}.  In the more general language of convex
geometry, it is the radial center of order $1$; radial centers were
introduced and studied by Moszy\'nska \cite{Moszynska} and others.

If $P$ is the maximizing point, Abraham and Kova\v{c} obtained the
equivalent characterization
\[
 {
 \frac1a\log\frac{PB+PC+a}{PB+PC-a}
 =
 \frac1b\log\frac{PC+PA+b}{PC+PA-b}
 =
 \frac1c\log\frac{PA+PB+c}{PA+PB-c}.
 }
\]
Thus the radial center has a particularly simple interpretation in
the present setting: it is the point through which the average chord
of the triangle is longest.

\section{A connection with dual quermassintegrals}

We finish by placing the preceding elementary problem in a broader
context.  If a star body $K\subset\mathbb R^2$ has radial function
$\rho_K$, one standard normalization of the dual quermassintegrals is
\[
 \widetilde W_i(K)
 =
 \frac12\int_0^{2\pi}\rho_K(\theta)^{\,2-i}\,d\theta.
\]
These quantities belong to the dual Brunn--Minkowski theory developed
by Lutwak; see \cite{Lutwak}.

Taking $K=T-P$ and $i=1$ gives
\[
 {
 M_T(P)=\frac{2}{\pi}\widetilde W_1(T-P).
 }
\]
Thus the average chord considered here is, up to normalization, the
first dual quermassintegral of the triangle translated so that $P$ is
the origin.

The elementary symmetrization argument used above also extends to
other radial powers.  For example, if $G$ is the centroid and
$0<q\le2$, the quantity
\[
 \int_0^{2\pi}\rho_G(\theta)^q\,d\theta
\]
is maximized, among triangles of fixed perimeter, by the equilateral
triangle.  Indeed,
\[
 \int_0^{2\pi}\rho_G(\theta)^q\,d\theta
 =
 q\int_T |X-G|^{q-2}\,dA(X),
\]
and the kernel on the right is radially nonincreasing.  Steiner
symmetrization preserves the centroid, since it preserves the lengths
of all parallel sections and moves their midpoints to the symmetry
axis.  Degenerate triangles cause no difficulty, since H\"older's
inequality gives
\[
 \int_0^{2\pi}\rho_G^q\,d\theta
 \le
 (2\pi)^{1-q/2}(2\area)^{q/2}.
\]
The case $q=1$ is the average-chord result, while $q=2$ is equivalent
to the familiar maximal-area property of the equilateral triangle.
We do not pursue these extensions here.

\section*{Acknowledgment}

The author gratefully acknowledges the assistance of ChatGPT, developed by
OpenAI, in exploring the problem, checking calculations, conducting numerical
experiments, reviewing related literature, and improving the exposition.
The author is responsible for the final mathematical statements and presentation.

\end{document}